\documentclass[11 pt]{amsart}
\usepackage{amsfonts}
\usepackage{ifthen}
\usepackage{amsthm}
\usepackage{amsmath}
\usepackage{graphicx}
\usepackage{amscd,amssymb,amsthm}
\usepackage{graphicx}
\usepackage{epstopdf}
\usepackage{hyperref}
\usepackage{cleveref}
\usepackage{cite}
\usepackage{mathrsfs}

\newcounter{minutes}
\divide\time by 60
\newcounter{hours}
\multiply\time by 60 \addtocounter{minutes}{-\time}

\newtheorem{lemma}{Lemma}[section]
\newtheorem{theorem}{Theorem}[section]
\newtheorem{corollary}{Corollary}[section]

\newcommand{\real}{\operatorname{Re}}

\newcommand{\loga}{\operatorname{Log}}


\begin{document}

\title{Radii of starlikeness and convexity of $q-$Mittag--Leffler functions }

\author[E. Toklu]{Evr{\.I}m Toklu}
\address{Department of Mathematics, Faculty of Education, A\u{g}r{\i} {\.I}brah{\i}m \c{C}e\c{c}en University, 04100 A\u{g}r{\i}, Turkey} 
\email{evrimtoklu@gmail.com}

\def\thefootnote{}
\footnotetext{ \texttt{File:~\jobname .tex,
printed: \number\year-0\number\month-\number\day,
\thehours.\ifnum\theminutes<10{0}\fi\theminutes}
} \makeatletter\def\thefootnote{\@arabic\c@footnote}\makeatother

\maketitle

\begin{abstract}
In this paper we deal with the radii of starlikeness and convexity of the $q-$Mittag--Leffler function for three different kinds of normalization by making use of their Hadamard factorization in such a way that the resulting functions are analytic in the unit disk of the complex plane. By applying Euler-Rayleigh inequalities for the first positive zeros of these functions tight lower and upper bounds for the radii of starlikeness of these functions are obtained. The Laguerre-P\'olya class of real entire functions plays a pivotal role in this investigation.
\end{abstract}

\section{\bf Introduction and the main results}
Frank Hilton Jackson, who was an English mathematician, studied what today is known as the $q-$calculus. In particular, he investigated some $q-$functions and the $q-$analogs of derivative and integral \cite{Jackson1}. In spite of the fact that Jackson started his studies introducing the $q-$difference operator, it is possible to say that this $q-$difference operator goes back to Euler and may go back to Heine and is reintroduced by Jackson in \cite{Jackson1}. For this reason, the $q-$difference operator is also called as the Euler-Heine-Jackson operator. After these studies, the $q-$calculus have started to appear in a generalization of many subjects, such as hypergeometric series, complex analysis, and particle physics. Because of the vast potential of its applications in solving problems on physical, engineering and earth sciences,  there has been a vivid interest on $q-$calculus from the point of view of geometric function theory. In \cite{IMS} Ismail {\it et al.} introduced and investigated the generalized class of starlike functions by making use of the $q-$difference operator. Some interesting applications of $q-$calculus seen on geometric function theory can be found in \cite{AD},\cite{AS}, \cite{RS}, \cite{EToklu} and the refences therein. Recently, in \cite{AB} Akta\c{s} and Baricz determined bounds for radii of starlikeness of some $q-$Bessel functions. And also, in \cite{SB} Srivastava and Bansal studied on close-to-convexity of certain family of $q-$Mittag-Leffler functions.

The gist of the present investigation is to determine, by using the method of Baricz {\it et al.} (see \cite{BKS}, \cite{BS}, \cite{BS1}), the radii of starlikeness and convexity of $q-$Mittag-Leffler function. Moreover, some intriguing applications of the technique of Baricz, which gives us a much simpler approach on determining the some geometric properties of special functions, can be found in \cite{ABO}, \cite{ATO}, \cite{AP},\cite{BTK}, \cite{Toklu}, \cite{Toklu1}, \cite{TAO} and the references therein. 
 
Before starting to present our main results, we would like to draw attention to some basic concepts needed for building our main results. For $r>0$ we denote by $\mathbb{D}_r=\left\{z\in\mathbb{C}: |z|<r\right\}$ the open disk of radius $r$ centered at the origin. Let $f:\mathbb{D}_r\to\mathbb{C}$ be the function defined by
\begin{equation}
f(z)=z+\sum_{n\geq 2}a_{n}z^{n},  \label{eq0}
\end{equation}
where $r$ is less or equal than the radius of convergence of the above power series. Denote by $\mathcal{A}$ the class of allanalytic functions of the form Eqn. \eqref{eq0}, that is, normalized by the conditions $f(0)=f^{\prime}(0)-1=0.$ We say that the function $f,$ defined by Eqn. \eqref{eq0}, is starlike function in $\mathbb{D}_r$ if $f$ is univalent in $\mathbb{D}_r$, and the image domain $f(\mathbb{D}_r)$ is a starlike domain in $\mathbb{C}$ with respect to the origin (see \cite{Dur} for more details). Analytically, the function $f$ is starlike in $\mathbb{D}_r$ if and only if $$\real\left( \frac{zf^{\prime }(z)}{f(z)}\right) >0 \quad \mbox{for all}\ \ z\in\mathbb{D}_r.$$ For $\alpha \in [0,1)$ we say that the function $f$ is starlike of order $\alpha $ in $\mathbb{D}_r$ if and only if 
$$\real\left( \frac{zf^{\prime }(z)}{f(z)}\right) >\alpha \quad \mbox{for all}\ \ z\in\mathbb{D}_r.$$
The radius of starlikeness of order $\alpha$ of the function $f$ is defined as the real number
\begin{equation*}
r_{\alpha }^{\star}(f)=\sup\left\{r>0\left|\real\left(\frac{zf^{\prime }(z)}{f(z)}\right)  >\alpha \;\text{for all }z\in	\mathbb{D}_r\right.\right\}.
\end{equation*}
Note that $r^{\star}(f)=r_{0}^{\star}(f)$ is in fact the largest radius such that the image region $f(\mathbb{D}_{r^{\star}(f)})$ is a starlike domain with respect to the origin.
The function $f,$ defined by Eqn. \eqref{eq0}, is convex in the disk $\mathbb{D}_r$ if $f$ is univalent in $\mathbb{D}_r$, and the image domain $f(\mathbb{D}_r)$ is a convex domain in $\mathbb{C}.$ Analytically, the function $f$ is convex in $\mathbb{D}_r$ if and only if
$$\real\left(  1+\frac{zf^{\prime \prime }(z)}{f^{\prime }(z)}\right)>0  \quad \mbox{for all}\ \ z\in\mathbb{D}_r.$$
For $\alpha \in[0,1)$ we say that the function $f$ is convex of order $\alpha $ in $\mathbb{D}_r$ if and only if 
$$\real\left( 1+\frac{zf^{\prime \prime }(z)}{f^{\prime }(z)}\right)
>\alpha \quad \mbox{for all}\ \ z\in\mathbb{D}_r.$$ 
We shall denote the radius of convexity of order $\alpha $ of the function $f$ by the real number
\begin{equation*}
r_{\alpha }^{c}(f)=\sup \left\{ r>0 \left|\real\left( 1+
\frac{zf^{\prime \prime }(z)}{f^{\prime }(z)}\right) >\alpha \;\text{for all }z\in\mathbb{D}_r\right.\right\} .
\end{equation*}
Note that $r^{c}(f)=r_{0}^{c}(f)$ is the largest radius such that the image region $f(\mathbb{D}_{r^{c}(f)})$ is a convex domain.

We recall that a real entire function $q$ belongs to the  Laguerre-P\'{o}lya class $\mathcal{LP}$ if it can be represented in the form $$q(x)=cx^{m}e^{-ax^2+bx}\prod_{n\geq1}\left(1+\frac{x}{x_n}\right)e^{-\frac{x}{x_n}},$$ with $c,b,x_n\in\mathbb{R}, a\geq0, m\in\mathbb{N}_0$ and $\sum\frac{1}{{x_n}^2}<\infty.$ We note that the class $\mathcal{LP}$ is the complement of the space of polynomials whose zeros are all real in the topology induced by the uniform convergence on the compact sets of the complex plane of polynomials with only real zeros. For more details on the class $\mathcal{LP}$ we refer to \cite[p. 703]{DC} and to the references therein.

\subsection{The three parameter generalization of the $q-$Mittag-Leffler function} First of all, we note that throughout of this paper, unless otherwise stated, $q$ is a positive number less than $1$ and by the word “basic” we mean a $q-$analogue. Now, let us consider the function $E_{\alpha,\beta}(z;q)$, which is called as $q-$Mittag-Leffler function,  defined by
\[E_{\alpha,\beta}(z;q)=\sum_{n\geq 0}\frac{q^{\alpha n(n-1)/2}}{\Gamma_{q}(n\alpha+\beta)}z^n, \quad (z\in \mathbb{C}, \alpha>0, \beta\in\mathbb{C})\]
where $\Gamma_{q}$ is the $q-$gamma function defined for $z\in\mathbb{C}-\left\lbrace-n :n\in \mathbb{N}_{0} \right\rbrace $ by
\[\Gamma_{q}(z)=\frac{(q;q)_{\infty}}{(q^{z};q)_{\infty}}(1-q)^{1-z}, \quad 0<q<1\]
and 
\[(a;q)_{0}=1, (a;q)_{n}=\prod_{k=1}^{n}(1-aq^{k-1}), (a;q)_{\infty}=\prod_{k\geq 1}(1-aq^{k-1}).\]
It is worthy to mention that the $q-$gamma function was introduced by Thomae \cite{Thomae} and later by Jackson \cite{Jackson}. Since $\Gamma_{q}(z)$ has no zeros, then $1/\Gamma_{q}(z)$ is entire function with zeros at $z=-n, n\in \mathbb{N}_{0}.$ It is clear that
\[\Gamma_{q}(n)=\frac{(q;q)_{n-1}}{(1-q)^{n-1}} \quad n\in \mathbb{N}.\]
Moreover, It is well known that for $x>0,$ $\Gamma_{q}(x)$ is the unique logarithmically convex function that satisfies the following relations
\[\Gamma_{q}(x+1)=\frac{1-q^x}{1-q}\Gamma_{q}(x),\quad \Gamma_{q}(1)=1.\]
For more historical remarks about the $q-$gamma function and its intriguing applications one can consult on \cite{AAR}, \cite{AM}, \cite{Ernst} and \cite{GR}.

We know that the function $z\mapsto E_{\alpha,\beta}(z;q)$ has infinitely many zeros. In \cite{AM} the authors was proven that for specific values of $\alpha$ and $\beta,$ $ E_{\alpha,\beta}(z;q),$ $0<q<1,$ may have only a finite number of non-real zeros. Moreover, if $q$ satisfies an additional conditions then the zeros of the function $z\mapsto E_{\alpha,\beta}(z;q)$ are all real. For more detail one can refer to \cite{Mansour} and \cite{Osler}.

Let us consider the function
\[\mathscr{E}_{\gamma,\sigma}(z;q)=E_{2,\gamma+1}(-\sigma^2 z;q), \quad (z\in \mathbb{C}),\]
where $\sigma$ is a fixed positive number and $0\leq \gamma<2.$ It is obvious that the function $z\mapsto\mathscr{E}_{\gamma,\sigma}(z;q)$ is of the form
\begin{equation}\label{qML}
\mathscr{E}_{\gamma,\sigma}(z;q)=\sum_{n\geq 0}\frac{(-1)^n\sigma^{2n}q^{n(n-1)}}{\Gamma_{q}(2n+\gamma+1)}z^n.
\end{equation}
It is worthy to mention that we have the following relations
\[\mathscr{E}_{0,\sigma}(z^2;q)=\cos(q^{-\frac{1}{2}}\sigma z;q) \text{ \ and \ } \mathscr{E}_{1,\sigma}(z^2;q)=\frac{\sin(q^{-1}\sigma z;q)}{q^{-1}\sigma z},\]
where $\sin(.;q)$ and $\cos(.;q)$ stand for the $q-$trigonometric functions. For some interesting applications of  $q-$trigonometric functions one can consult on \cite{AM} and \cite{Ernst}. And also, we note that Annabay and Mansour \cite[see Chapter 2]{AM} proved that the zeros of the functions $\cos (.;q)$ and $\sin (.;q)$ are real and simple.

We note that throughout this investigation, we shall focus on the function $z\mapsto \mathscr{E}_{\gamma,\sigma}(z;q)$ defined by \eqref{qML}. It is easy to check that the function $z \mapsto \mathscr{E}_{\gamma,\sigma}(z^2;q)$ is not of the class $\mathcal{A}.$ Thus first we shall perform some natural normalization. We define three functions originating from $\mathscr{E}_{\gamma,\sigma}(.;q):$
\begin{align*}
f_{\gamma,\sigma}(z;q)&=\left(z^{\gamma+1}\Gamma_{q}(\gamma+1) \mathscr{E}_{\gamma,\sigma}(z^2;q) \right)^{\frac{1}{\gamma+1}}, \\
g_{\gamma,\sigma}(z;q&=z\Gamma_{q}(\gamma+1)\mathscr{E}_{\gamma,\sigma}(z^2;q),\\
h_{\gamma,\sigma}(z;q)&=z\Gamma_{q}(\gamma+1)\mathscr{E}_{\gamma,\sigma}(z;q).
\end{align*}
It is obvious that each of these functions are of the class $\mathcal{A}.$ Of course, it can be written infinitely many other normalization; the main motivation to consider the above ones is the studied normalization in the literature of Bessel, $q-$Bessel, Mittag-Leffler, Struve, Lommel and Wright functions. Moreover, it is worth mentioning here that in fact 
\[f_{\gamma,\sigma}(z;q)=\exp\left[\frac{1}{\gamma+1}\loga(z^{\gamma+1}\Gamma_{q}(\gamma+1)\mathscr{E}_{\gamma,\sigma}(z^2;q))\right], \]
where $\loga$ represents the principle branch of the logarithm function and every many-valued function considered in this paper are taken with the principal branch.

The following lemma, which characterize the reality of zeros of the function $z\mapsto \mathscr{E}_{\gamma,\sigma}(.;q),$  take a leading part in building up our main results. For some results about the zeros of some $q-$functions one can refer to \cite{AM}, \cite{AM1} and the references therein.
\begin{lemma}\cite[p. 220]{AM} \label{Lemma1}
Let $0 \leq\gamma<2.$ Then
\begin{itemize}
	\item [\bf 1.] If $q$ satisfies the condition
		\[q^{-1}(1-q)(1-q^{\gamma+1})(1-q^{\gamma+2})>1,\text{ \ \ } \gamma \in (0,2), \quad \gamma \neq 1\] 
		then the zeros of $z \mapsto \mathscr{E}_{\gamma,\sigma}(z^2;q) $ are all real, simple, symmetric and its positive zeros lie in the intervals, for $n\in \mathbb{N}$
		\begin{align*}
		\varepsilon_{\gamma,\sigma,n}(q) \in \left\{ \begin{array}{cc} 
		\left(\frac{q^{-n+\frac{3}{2}}\sqrt{(1-q^{\gamma+1})(1-q^{\gamma+2})}}{\sigma(1-q)}, \frac{q^{-n+\frac{1}{2}}\sqrt{(1-q^{\gamma+1})(1-q^{\gamma+2})}}{\sigma(1-q)} \right),  & \gamma\in (0,1), \\
		\left(\frac{q^{-n+\frac{5}{2}}\sqrt{(1-q^{\gamma+1})(1-q^{\gamma+2})}}{\sigma(1-q)},\frac{q^{-n+\frac{3}{2}}\sqrt{(1-q^{\gamma+1})(1-q^{\gamma+2})}}{\sigma(1-q)}\right),  & \gamma\in (1,2), \\
		\end{array} \right.
		\end{align*}
		one zero in each interval.
		\item [\bf 2.] If $\gamma \in \left\lbrace 0,1\right\rbrace $ then $\mathscr{E}_{\gamma,\sigma}(z^2;q) $ has only real, simple and symmetric zeros such that
		\[\varepsilon_{0,\sigma,n}(q)=\frac{q^{\frac{1}{2}}x_m}{\sigma} \text{ \ and \ } \varepsilon_{1,\sigma,n}(q)=\frac{qy_m}{\sigma} \quad (m\in \mathbb{N}), \]
		where $x_m$ and $y_m$ are, respectively, the positive zeros of the functions $\cos(z;q)$ and $\sin(z;q).$
\end{itemize}
\end{lemma}
The following lemma,  which we believe is of independent interest, plays a pivotal role in proving our main results which are related to radii of starlikeness and convexity of functions $f_{\gamma,\sigma}(z;q)$, $g_{\gamma,\sigma}(z;q)$, and $h_{\gamma,\sigma}(z;q).$
\begin{lemma} \label{q-Mittag-LefflerLemma}
	Let $\sigma$ is a fixed positive real number,  $0\leq \gamma <2.$ Moreover, under the conditions of \Cref{Lemma1} the function $z \mapsto \mathscr{E}_{\gamma,\sigma}(z^2;q) $ has infinitely many zeros which are all real. Denoting by $\varepsilon_{\gamma,\sigma,n}(q)$ the $n$th positive zero of $z\mapsto\mathscr{E}_{\nu,c}(z^2;q)$, under the same conditions the Weierstrassian decomposition
	\begin{equation}\label{InfiniteProductq-ML}
	\mathscr{E}_{\gamma,\sigma}(z^2 ;q)=\frac{1}{\Gamma_{q}(\gamma+1)}\prod_{n\geq 1}\left(1-\frac{z^2}{\varepsilon_{\gamma,\sigma,n}^2(q)}\right) 
	\end{equation}
	is fulfilled, and this product is uniformly convergent on compact subsets of the complex plane. Moreover, if we denote by $\xi_{\gamma,\sigma,n}(q)$ the nth positive zero of $\Psi_{\gamma,\sigma}^{\prime}(z;q)$, where $\Psi_{\gamma,\sigma}(z;q)=z^{\gamma+1}\mathscr{E}_{\gamma,\sigma}(z^2;q)$, then positive zeros of $\varepsilon_{\gamma,\sigma,n}(q)$ and $\xi_{\gamma,\sigma,n}(q)$ are interlaced. In other words, the zeros satisfy the chain of inequalities	$$\xi_{\gamma,\sigma,1}(q)<\varepsilon_{\gamma,\sigma,1}(q)<\xi_{\gamma,\sigma,2}(q)<\varepsilon_{\gamma,\sigma,2}(q)<\dots_{.}$$
\end{lemma}
\begin{proof}
As a clearly stated in \Cref{Lemma1} the function $z \mapsto \mathscr{E}_{\gamma,\sigma}(z^2;q) $ has real zeros under the condition of the Lemma. Next, we need to calculate the growth order of the  function $\mathscr{E}_{\gamma,\sigma}(z^2;q).$ We have
\[\rho(\mathscr{E}_{\gamma,\sigma}(z^2;q))=\limsup_{n\rightarrow \infty}\frac{n\log n}{-\log \left| c_n\right| }\]
where $c_n$ stands for the coefficient of $z^{2n}$ stated in \eqref{qML}, i.e, 
\[c_n=\frac{(-\sigma^2)^n q^{n(n-1)}}{\Gamma_{q}(2n+\gamma+1)}, \quad (n \in \mathbb{N}_{0}).\]
Hence
\[-\log\left| c_n\right|=-n(n-1)\log q -2n\log \sigma +\log\left|\Gamma_{q}(2n+\gamma+1) \right|. \]
Making use of (Denklem numarası) and the definition of the $q-$gamma function, c.f. (denklem numarası) we get
\begin{align}
\log\left|\Gamma_{q}(2n+\gamma+1) \right|&=\log\left|\frac{(q;q)_{\infty}}{\left( q^{2n+\gamma+1};q\right)_{\infty}}(1-q)^{-2n-\gamma} \right| \nonumber\\
&=\log (q;q)_{\infty} -(2n+\gamma)\log (1-q) -\log\left|(q^{2n+\gamma+1};q)_{\infty} \right|, 
\end{align}
and
\begin{align}
\log\left|(q^{2n+\gamma+1};q)_{\infty} \right|&=\log\left( \prod_{k\geq0}\left|1-q^{2n+\gamma+k+1} \right| \right)=\log\left(\lim_{m\rightarrow \infty}\prod_{k=0}^{m} \left|1-q^{2n+\gamma+k+1} \right|\right)  \nonumber\\
&=\lim_{m\rightarrow \infty}\sum_{k=0}^{m}\log \left|1-q^{2n+\gamma+k+1} \right|=\sum_{k\geq 0}\log\left|1-q^{2n+\gamma+k+1} \right|\nonumber.
\end{align}
Since
\[\log\left|1-q^{2n+\gamma+k+1} \right|\leq \log\left(1+ \left|q^{2n+\gamma+k+1} \right|\right)\leq \left|q^{2n+\gamma+k+1} \right|=q^{2n+\gamma+k+1},\]
we obtain
\[\sum_{k\geq 0}\log\left|1-q^{2n+\gamma+k+1} \right|\leq \sum_{k\geq 0}q^{2n+\gamma+k+1}=\frac{q^{2n+\gamma+1}}{1-q}.\]
Consequently, we get
\[\lim_{n\rightarrow \infty}\frac{\log\left| (q^{2n+\gamma+1};q)_{\infty}\right| }{n \log n}=0,\]
which implies that
\[\lim_{n\rightarrow \infty}\frac{\log\left|\Gamma_{q}(2n+\gamma+1) \right| }{n\log n}=0.\]
Taking into account that
\[\lim_{n\rightarrow \infty}\frac{n-1}{\log n}=\infty,\]
we obtain
\[\lim_{n\rightarrow \infty}\frac{n\log n}{-\log\left| c_n\right| }=0,\]
which implies that \(\rho(\mathscr{E}_{\gamma,\sigma}(z^2;q))=0.\) Moreover, It is well known that the finite growth order $\rho$ of an entire function is not equal to a positive integer, then  the function has infinitely many zeros. That is to say, the function $\mathscr{E}_{\gamma,\sigma}(z^2;q)$ given in \eqref{qML} has infinitely zeros which are all real and simple. In this case, by virtue of the Hadamard theorem on growth order of the entire function, it follows that its infinite product representation is exactly what we have in \Cref{q-Mittag-LefflerLemma}. This means that the function $\mathscr{E}_{\gamma,\sigma}(z^2;q)$ belongs to the Laguerre-P\'olya class $\mathcal{LP}$ of entire functions. As a natural consequence of this, we deduce that the function $z\mapsto \Psi_{\gamma,\sigma}(z;q)$ belongs also to the Laguerre-P\'olya class $\mathcal{LP}.$ Since  $\mathcal{LP}$ is closed differentiation the function $z\mapsto \Psi_{\gamma,\sigma}^{\prime}(z;q)$ belongs also to the class $\mathcal{LP}.$ Hence the function $z\mapsto\Psi_{\gamma,\sigma}^{\prime}(z;q)$ has only real zeros under the same conditions. It is important to mention that throughout of this paper for the sake of simplicity, we use the notation $\lambda_{\gamma,\sigma}(z;q)=\mathscr{E}_{\gamma,\sigma}(z^2;q).$ Now, with the aid of the infinite product representation we get
\begin{equation}\label{MainLemmaEq1}
\frac{\Psi_{\nu,c}^{\prime}(z;q)}{\Psi_{\nu,c}(z;q)}=\frac{\gamma+1}{z}+\frac{\lambda_{\gamma,\sigma}'(z;q)}{\lambda_{\gamma,\sigma}(z;q)}=\frac{\gamma+1}{z}+\sum_{n\geq 1}\frac{2z}{z^{2}-\varepsilon_{\gamma,\sigma,n}^{2}(q)}.
\end{equation}
Differentiating both sides of Eq. \eqref{MainLemmaEq1}, we arrive at
\[\frac{d}{dz}\left( \frac{\Psi_{\nu,c}^{\prime}(z;q)}{\Psi_{\nu,c}(z;q)}\right)=-\frac{\gamma+1}{z^2}-2\sum_{n\geq 1}\frac{z^2+\varepsilon_{\gamma,\sigma,n}^{2}(q)}{\left( z^2-\varepsilon_{\gamma,\sigma,n}^{2}(q)\right)^2 }, \quad z\neq \varepsilon_{\gamma,\sigma,n}(q). \]
It is clear that the expression on the right-hand side is real and negative for the same assumptions of the lemma. That is to say, for each real $z,$ $\tfrac{\Psi_{\nu,c}^{\prime}(z;q)}{\Psi_{\nu,c}(z;q)}<0$ which implies that the quotient $\frac{\Psi_{\nu,c}^{\prime}}{\Psi_{\nu,c}}$ is a strictly decreasing function from $+\infty$ to $-\infty$ as $z$ increases through real values over the open interval $\left(\varepsilon_{\gamma,\sigma,n}(q), \varepsilon_{\gamma,\sigma,n+1}(q) \right),$ $n\in \mathbb{N}.$ Hence between any two zeros of the function $\lambda_{\gamma,\sigma}(z;q)$ there must be precisely one of $\Psi_{\gamma,\sigma}^{\prime}(z;q).$
\end{proof}
\subsection{Radii of starlikeness of the $q-$Mittag-Leffler functions} This section is devoted to determine the radii of starlikeness of the normalized forms of the $q-$Mittag-Leffler functions, that is of $f_{\gamma,\sigma}(z;q),$ $g_{\gamma,\sigma}(z;q)$ and $h_{\gamma,\sigma}(z;q).$ In addition, in this section we aim to give some tight lower and upper bounds for the radii of starlikeness and convexity of order zero for these functions.

\begin{theorem}\label{Theo1}
Let $\alpha \in \left[0,1 \right) $ and with the conditions of \Cref{q-Mittag-LefflerLemma} the following assertions hold true:
\begin{itemize}
	\item [\bf a.] The radius of starlikeness of order $\alpha$ of the function  $f_{\gamma,\sigma}$ is $r^{\star}_{\alpha}( f_{\gamma,\sigma}(z;q))=x_{\gamma,\sigma,1}(q)$, where $x_{\gamma,\sigma,1}(q)$ stands for the smallest positive zero of the equation
	\[r\lambda_{\gamma,\sigma}'(r;q)-(\gamma+1)(\alpha-1)\lambda_{\gamma,\sigma}(r;q)=0.\]
	\item [\bf b.] The radius of starlikeness of order $\alpha$ of the function  $g_{\gamma,\sigma}$ is $r^{\star}_{\alpha}( g_{\gamma,\sigma}(z;q))=y_{\gamma,\sigma,1}(q)$, where $y_{\gamma,\sigma,1}(q)$ stands for the smallest positive zero of the equation
	\[r\lambda_{\gamma,\sigma}'(r;q)-(\alpha-1)\lambda_{\gamma,\sigma}(r;q)=0.\]
	\item [\bf c.] The radius of starlikeness of order $\alpha$ of the function  $h_{\gamma,\sigma}$ is $r^{\star}_{\alpha}( h_{\gamma,\sigma}(z;q))=z_{\gamma,\sigma,1}(q)$, where $z_{\gamma,\sigma,1}$ stands for the smallest positive zero of the equation
	\[\sqrt{r}\lambda_{\gamma,\sigma}'(\sqrt{r};q)-2(\alpha-1)\lambda_{\gamma,\sigma}(\sqrt{r};q)=0.\]
\end{itemize}
\end{theorem}
\begin{proof}
We need to show that the inequalities
\begin{equation}\label{Theo1Ineq1}
\real\left( \frac{zf_{\gamma,\sigma}'(z;q)}{f_{\gamma,\sigma}(z;q)}\right) \geq \alpha, \text{ \ \ }\real\left( \frac{zg_{\gamma,\sigma}'(z;q)}{g_{\gamma,\sigma}(z;q)}\right) \geq \alpha \text{ \ and \ } \real\left( \frac{zh_{\gamma,\sigma}'(z;q)}{h_{\gamma,\sigma}(z;q)}\right) \geq \alpha
\end{equation}
hold for $z\in \mathbb{D}_{r^{\star}_{\alpha}( f_{\gamma,\sigma})},$ $z\in \mathbb{D}_{r^{\star}_{\alpha}( g_{\gamma,\sigma})},$ and $z\in \mathbb{D}_{r^{\star}_{\alpha}( h_{\gamma,\sigma})},$ respectively, and each of the above inequalities does not hold in any larger disk. Recall that under the corresponding conditions the zeros of the $q-$Mittag-Leffler function $\mathscr{E}_{\gamma,\sigma}(z^2;q)$ are all real and simple. Hence in light of \Cref{q-Mittag-LefflerLemma} the $q-$Mittag-Leffler function has the infinite product representation given by
\[\mathscr{E}_{\gamma,\sigma}(z^2 ;q)=\frac{1}{\Gamma_{q}(\gamma+1)}\prod_{n\geq 1}\left(1-\frac{z^2}{\varepsilon_{\gamma,\sigma,n}^2(q)}\right)\] 
and this infinite product is uniformly convergent on each compact subset of $\mathbb{C}.$ Taking into account fact that we use the notation $\lambda_{\gamma,\sigma}(z;q)=\mathscr{E}_{\gamma,\sigma}(z^2 ;q)$, and by logarithmic differentiation we get
\[\frac{\lambda_{\gamma,\sigma}'(z;q)}{\lambda_{\gamma,\sigma}(z;q)}=-\sum_{n\geq 1}\frac{2z}{\varepsilon_{\gamma,\sigma,n}^2(q)-z^2},\]
which implies that
\[\frac{zf_{\gamma,\sigma}'(z;q)}{f_{\gamma,\sigma}(z;q)}=1-\frac{1}{\gamma+1}\sum_{n\geq 1}\frac{2z}{\varepsilon_{\gamma,\sigma,n}^2(q)-z^2}, \quad \frac{zg_{\gamma,\sigma}'(z;q)}{g_{\gamma,\sigma}(z;q)}=1-\sum_{n\geq 1}\frac{2z^2}{\varepsilon_{\gamma,\sigma,n}^2(q)-z^2}\] 
and
\[\frac{zh_{\gamma,\sigma}'(z;q)}{h_{\gamma,\sigma}(z;q)}=1-\sum_{n\geq 1}\frac{2z}{\varepsilon_{\gamma,\sigma,n}^2(q)-z}.\]
From \citen{BS}, we know that if $z\in \mathbb{C}$ and $\theta\in\mathbb{R}$ are such that $\theta>\left| z\right|,$ then
\begin{equation}\label{ClassicalIneq}
\frac{\left|z \right| }{\theta-\left| z\right| }\geq \real\left( \frac{z}{\theta}-z\right). 
\end{equation}
By virtue of the inequality \eqref{ClassicalIneq} we deduce that the inequality
\[\frac{\left| z\right| ^2}{\varepsilon_{\gamma,\sigma,n}^2(q)-\left| z\right| ^2}\geq \real\left(\frac{z^2}{\varepsilon_{\gamma,\sigma,n}^2(q)-z^2} \right) \]
holds under the conditions of \Cref{q-Mittag-LefflerLemma}, $n\in\mathbb{N}$ and $\left|z \right|<\varepsilon_{\gamma,\sigma,1}(q), $ and therefore under the same conditions we get
\begin{align*}
\real\left( \frac{zf_{\gamma,\sigma}(z;q)}{f_{\gamma,\sigma}(z;q)}\right)&=1-\frac{1}{\gamma+1}\real\left( \sum_{n\geq 1}\frac{2z^2}{\varepsilon_{\gamma,\sigma,n}^2(q)-z^2} \right) \\
&\geq1-\frac{1}{\gamma+1}\sum_{n\geq1}\frac{2\left|z\right|^2}{\varepsilon_{\gamma,\sigma,n}^2(q)-\left| z\right| ^2}=\frac{\left| z\right| f_{\gamma,\sigma}(\left| z\right| ;q)}{f_{\gamma,\sigma}(\left| z\right| ;q)},\\
\real\left( \frac{zg_{\gamma,\sigma}(z;q)}{g_{\gamma,\sigma}(z;q)}\right)&= 1-\real\left( \sum_{n\geq 1}\frac{2z^2}{\varepsilon_{\gamma,\sigma,n}^2(q)-z^2} \right) \geq 1-\sum_{n\geq 1}\frac{2\left| z\right| ^2}{\varepsilon_{\gamma,\sigma,n}^2(q)-\left| z\right| ^2}=\frac{\left| z\right| g_{\gamma,\sigma}(\left| z\right| ;q)}{g_{\gamma,\sigma}(\left| z\right| ;q)},\\
\real\left( \frac{zh_{\gamma,\sigma}(z;q)}{h_{\gamma,\sigma}(z;q)}\right)&=1-\real\left( \sum_{n\geq 1}\frac{z}{\varepsilon_{\gamma,\sigma,n}^2(q)-z} \right) \geq 1-\sum_{n\geq 1}\frac{\left| z\right|}{\varepsilon_{\gamma,\sigma,n}^2(q)-\left| z\right|}=\frac{\left| z\right| h_{\gamma,\sigma}(\left| z\right| ;q)}{h_{\gamma,\sigma}(\left| z\right| ;q)}
\end{align*}
where equalities occur only when $z=\left| z\right|=r.$ The minimum principle for harmonic functions and the previous inequalities imply that the corresponding inequalities given in \eqref{Theo1Ineq1} are valid if and only if we have $\left| z\right| <x_{\gamma,\sigma,1}(q),$ $\left| z\right| <y_{\gamma,\sigma,1}(q)$ and $\left| z\right| <z_{\gamma,\sigma,1}(q),$ respectively, where $x_{\gamma,\sigma,1}(q),$ $y_{\gamma,\sigma,1}(q)$ and $z_{\gamma,\sigma,1}(q)$ are the smallest positive roots of the following equalities
\[\real\left( \frac{rf_{\gamma,\sigma}'(r;q)}{f_{\gamma,\sigma}(r;q)}\right)= \alpha, \text{ \ \ }\real\left( \frac{rg_{\gamma,\sigma}'(r;q)}{g_{\gamma,\sigma}(r;q)}\right)= \alpha \text{ \ and \ } \real\left( \frac{rh_{\gamma,\sigma}'(r;q)}{h_{\gamma,\sigma}(r;q)}\right)=\alpha\]
which imlies that
\[r\lambda_{\gamma,\sigma}'(r;q)-(\gamma+1)(\alpha-1)\lambda_{\gamma,\sigma}(r;q)=0, \quad r\lambda_{\gamma,\sigma}'(r;q)-(\alpha-1)\lambda_{\gamma,\sigma}(r;q)=0\]
and
\[\sqrt{r}\lambda_{\gamma,\sigma}'(\sqrt{r};q)-2(\alpha-1)\lambda_{\gamma,\sigma}(\sqrt{r};q)=0.\]
\end{proof}
The following theorem gives some tight lower and upper bounds for the radii of starlikeness of the functions seen on the above theorem. 
\begin{theorem}\label{Starlikenessq-ML}
Let the conditions of \Cref{q-Mittag-LefflerLemma} remain valid.
\begin{itemize}
	\item [\bf a.] The radius of starlikeness $r^{\star}(f_{\gamma,\sigma}(z;q))$ satisfies the inequalities
	\[\text{\footnotesize $\frac{\sigma^2(\gamma+3)\Gamma_{q}(\gamma+1)}{(\gamma+1)\Gamma_{q}(\gamma+3)}-\frac{2\sigma^2q^2(\gamma+5)\Gamma_{q}(\gamma+3)}{(\gamma+3)\Gamma_{q}(\gamma+5)}<\left(r^{\star}(f_{\gamma,\sigma}(z;q))\right) ^{-2}<\frac{\sigma^2(\gamma+3)\Gamma_{q}(\gamma+1)}{(\gamma+1)\Gamma_{q}(\gamma+3)} $}.\]
	\item [\bf b.] The radius of starlikeness $r^{\star}(g_{\gamma,\sigma}(z;q))$ satisfies the inequalities
	\[ \text{\footnotesize $\frac{\Gamma_{q}(\gamma+3)}{3\sigma^2\Gamma_{q}(\gamma+1)}<\left( r^{\star}(g_{\gamma,\sigma}(z;q))\right)^{2}<\frac{3\Gamma_{q}(\gamma+3)\Gamma_{q}(\gamma+5)}{\sigma^{2}\left(9\Gamma_{q}(\gamma+1)\Gamma_{q}(\gamma+5)-10q^{2}\Gamma_{q}^{2}(\gamma+3)\right) }. $} \]
	\item [\bf c.] The radius of starlikeness $r^{\star}(h_{\gamma,\sigma}(z;q))$ satisfies the inequalities
	\[\frac{\Gamma_{q}(\gamma+3)}{2\sigma^2\Gamma_{q}(\gamma+1)}<r^{\star}(h_{\gamma,\sigma}(z;q))<\frac{\Gamma_{q}(\gamma+3)\Gamma_{q}(\gamma+5)}{\sigma^2\left(2\Gamma_{q}(\gamma+1)\Gamma_{q}(\gamma+5)-3q^2\Gamma_{q}^2(\gamma+3)\right) }.\]
\end{itemize}
\end{theorem}
\begin{proof}
\begin{itemize}
	\item [\bf a.] The radius of starlikeness of the normalized $q-$Mittag-Leffler function $f_{\gamma,\sigma}(z;q)$ coincide with the radius of starlikeness of the function $\Psi_{\gamma,\sigma}(z;q)=z^{\gamma+1}\lambda_{\gamma,\sigma}(z;q).$ The infinite series
	representation of the function $\Psi_{\gamma,\sigma}'(z;q)$ and its derivative are as follows:
	\begin{equation}\label{Psi1}
	\Psi_{\gamma,\sigma}'(z;q)=\sum_{n\geq 0}\frac{(-1)^{n}\sigma^{2n}(2n+\gamma+1)q^{n(n-1)}}{\Gamma_{q}(2n+\gamma+1)}z^{2n+\gamma}
	\end{equation}
	and
	\begin{equation}\label{Psi2}
	\Psi_{\gamma,\sigma}''(z;q)=\sum_{n\geq 0}\frac{(-1)^n\sigma^{2n}(2n+\gamma+1)(2n+\gamma)q^{n(n-1)}}{\Gamma_{q}(2n+\gamma+1)}z^{2n+\gamma-1}.
	\end{equation}
	By means of \Cref{q-Mittag-LefflerLemma} the function $z\mapsto\Psi_{\gamma,\sigma}(z;q)$ belongs to the Laguerre-P\'olya class $\mathcal{LP}.$ Because of the fact that this class of functions is closed under differentiation, $z\mapsto\Psi_{\gamma,\sigma}'(z;q)$ belong also to the Laguerre-P\'olya class $\mathcal{LP}.$ This means that the zeros of the function $z\mapsto\Psi_{\gamma,\sigma}'(z;q)$ are all real, and in fact due to \Cref{q-Mittag-LefflerLemma} they are interlaced with the zeros of $z\mapsto\Psi_{\gamma,\sigma}(z;q).$ Therefore, $\Psi_{\gamma,\sigma}'(z;q)$ can be represented by the infinite product form
	\begin{equation}\label{Theo2Eq1}
	\Psi_{\gamma,\sigma}'(z;q)=\frac{(\gamma+1)z^{\gamma}}{\Gamma_{q}(\gamma+1)}\prod_{n\geq 1}\left(1-\frac{z^2}{\xi_{\gamma,\sigma,n}^{2}(q)} \right).
	\end{equation}
	If we take the logarithmic derivative of both sides of \eqref{Theo2Eq1} for $\left|z \right|<\xi_{\gamma,\sigma,1}$ we obtain
	\begin{align}\label{Compere1}
	\frac{z\Psi_{\gamma,\sigma}''(z;q)}{\Psi_{\gamma,\sigma}'(z;q)}-\gamma&=-\sum_{n\geq 1}\frac{2z^2}{\xi_{\gamma,\sigma,n}^{2}(q)-z^2}=-2\sum_{n\geq 1}\sum_{k\geq 0}\frac{z^{2k+2}}{\xi_{\gamma,\sigma,n}^{2k+2}(q)} \nonumber \\
	&=-2\sum_{k\geq 0}\sum_{n\geq 1}\frac{z^{2k+2}}{\xi_{\gamma,\sigma,n}^{2k+2}(q)}=-2\sum_{k\geq 0}\kappa_{k+1}z^{2k+2}
	\end{align}
	where $\kappa_{k}=\sum_{n\geq 1}\xi_{\gamma,\sigma,n}^{-2k}(q).$ On the other hand, by making use of \eqref{Psi1} and \eqref{Psi2} we obtain
	\begin{equation}\label{Compere2}
	\frac{z\Psi_{\gamma,\sigma}''(z;q)}{\Psi_{\gamma,\sigma}'(z;q)}=\sum_{n\geq 0} a_{n}z^{2n}\bigg/ \sum_{n\geq 0}b_{n}z^{2n},
	\end{equation}
	where
	\[a_n=\frac{(-1)^n\sigma^{2n}(2n+\gamma+1)(2n+\gamma)q^{n(n-1)}}{\Gamma_{q}(2n+\gamma+1)} \text{ \ and \ } b_n=\frac{(-1)^{n}\sigma^{2n}(2n+\gamma+1)q^{n(n-1)}}{\Gamma_{q}(2n+\gamma+1)}.\]
	By comparing the coefficients of \eqref{Compere1} and \eqref{Compere2} we get
	\[a_0=\gamma b_{0}, \text{ \ \ }a_1=\gamma b_{1}-2b_{0}\kappa_{1}, \text{ \ \ } a_2=\gamma b_{2}-2b_{1}\kappa_{1}-2b_{0}\kappa_{2},\]
	which gives the following Rayleigh sums
	\[\kappa_{1}=\frac{\sigma^{2}(\gamma+3)\Gamma_{q}(\gamma+1)}{(\gamma+1)\Gamma_{q}(\gamma+3)} \text{ \ and \ } \kappa_{2}=\frac{\sigma^{4}(\gamma+3)^{2}\Gamma_{q}^{2}(\gamma+1)}{(\gamma+1)^{2}\Gamma_{q}^{2}(\gamma+3)}-\frac{2\sigma^{4}q^{2}(\gamma+5)\Gamma_{q}(\gamma+1)}{(\gamma+1)\Gamma_{q}(\gamma+5)}.\]
	By using the Euler-Rayleigh inequalities	\[\kappa_{k}^{\frac{-1}{k}}<\xi_{\gamma,\sigma,1}^2<\frac{\kappa_{k}}{\kappa_{k+1}}\]
	for $k=1$ we have
	\[\frac{\sigma^2(\gamma+3)\Gamma_{q}(\gamma+1)}{(\gamma+1)\Gamma_{q}(\gamma+3)}-\frac{2\sigma^2q^2(\gamma+5)\Gamma_{q}(\gamma+3)}{(\gamma+3)\Gamma_{q}(\gamma+5)}<\left(r^{\star}(f_{\gamma,\sigma}(z;q))\right) ^{-2}<\frac{\sigma^2(\gamma+3)\Gamma_{q}(\gamma+1)}{(\gamma+1)\Gamma_{q}(\gamma+3)},\]
	which is desired result.

	\item [\bf b.] If we take $\alpha=0$ in the second part of the \Cref{Starlikenessq-ML}, then we have that the radius of starlikeness of order zero	of the function $g_{\gamma,\sigma}(z;q)$ is the smallest positive root of the equation $(z\lambda_{\gamma,\sigma}(z;q))'=0.$ Therefore, we shall study the first positive zero of
	\begin{equation}\label{Theo2Partb1}
	\phi_{\gamma,\sigma}(z;q)=(z\lambda_{\gamma,\sigma}(z;q))'=\sum_{n\geq 0}\frac{(-1)^{n}\sigma^{2n}(2n+1)q^{n(n-1)}}{\Gamma_{q}(2n+\gamma+1)}z^{2n}.
	\end{equation}
	We know that the function $z\mapsto \lambda_{\gamma,\sigma}(z;q)$ belongs to the Laguerre-P\'olya class $\mathcal{LP},$ which is closed under differentiation. Therefore, we get that the function $z\mapsto\phi_{\gamma,\sigma}(z;q)$ belongs to the Laguerre-P\'olya class, and hence all its zeros are real. Let denote $\theta_{\gamma,\sigma,n}(q)$ the $n$th positive zero of $z\mapsto\phi_{\gamma,\sigma}(z;q).$ Since growth order of the function $z\mapsto\phi_{\gamma,\sigma}(z;q)$ coincides with the growth order of the $q-$Mittag-Leffler function itself, it can be written as
	\begin{equation}\label{Theo2PhiProd}
	\phi_{\gamma,\sigma}(z;q)=\frac{1}{\Gamma_{q}(\gamma+1)}\prod_{n\geq 1}\left( 1-\frac{z^2}{\theta_{\gamma,\sigma,n}^2(q)}\right).
	\end{equation}
	Logarithmic differentiation of both sides of \eqref{Theo2PhiProd} for $\left| z\right| <\theta_{\gamma,\sigma,1}(q)$ gives
	\begin{equation}\label{The2ParbCoeff1}
	\text{\footnotesize $ \frac{\phi_{\gamma,\sigma}'(z;q)}{\phi_{\gamma,\sigma}(z;q)}=\sum_{n\geq 1}\frac{-2z}{\theta_{\gamma,\sigma,n}^2(q)-z^2}=\sum_{n\geq 1}\sum_{k\geq 0}\frac{-2z^{2k+1}}{\theta_{\gamma,\sigma,n}^{2k+2}(q)}=\sum_{k\geq 0}\sum_{n\geq 1}\frac{-2z^{2k+1}}{\theta_{\gamma,\sigma,n}^{2k+2}(q)}=-2\sum_{k\geq 0}\chi_{k+1}z^{2k+1},$}	
	\end{equation}
	where $\chi_{k}=\sum_{n\geq 1}\theta_{\gamma,\sigma,n}^{-2k}(q).$ Moreover, with the aid of \eqref{Theo2Partb1} we get
	\begin{equation}\label{The2ParbCoeff2}
	\frac{\phi_{\gamma,\sigma}'(z;q)}{\phi_{\gamma,\sigma}(z;q)}=-2\sum_{n\geq 0}c_nz^{2n+1} \bigg/ \sum_{n\geq 0}d_nz^{2n},
	\end{equation}
	where
	\[c_n= \frac{(-1)^{n}\sigma^{2n+2}(n+1)(2n+3)q^{n(n+1)}}{\Gamma_{q}(2n+\gamma+3)}\text{ \ and \ } d_n=\frac{(-1)^{n}\sigma^{2n}(2n+1)q^{n(n-1)}}{\Gamma_{q}(2n+\gamma+1)}.  \]
	Comparing the coefficients in \eqref{The2ParbCoeff1} and \eqref{The2ParbCoeff2} we have that 
	\[d_{0}\chi_{1}=c_{0}\text{ \ and \ } d_{0}\chi_{2}+d_{1}\chi_{1}=c_{1},\]
	which yields the following Rayleigh sums
	\[\chi_{1}=\frac{3\sigma^{2}\Gamma_{q}(\gamma+1)}{\Gamma_{q}(\gamma+3)} \text{ \ and \ } \chi_{2}=\frac{9\sigma^{4}\Gamma_{q}^2(\gamma+1)}{\Gamma_{q}^2(\gamma+3)}-\frac{10\sigma^{4}q^{2}\Gamma_{q}(\gamma+1)}{\Gamma_{q}(\gamma+5)}.\]
	By making use of the Euler-Rayleigh inequalities
	\[\chi_{k}^{-\frac{1}{k}}<\theta_{\gamma,\sigma,1}^{2}(q)<\frac{\chi_{k}}{\chi_{k+1}}\]
	for $k=1$ we obtain 
	\[\frac{\Gamma_{q}(\gamma+3)}{3\sigma^2\Gamma_{q}(\gamma+1)}<\left( r^{\star}(g_{\gamma,\sigma}(z;q))\right)^{2}<\frac{3\Gamma_{q}(\gamma+3)\Gamma_{q}(\gamma+5)}{\sigma^{2}\left(9\Gamma_{q}(\gamma+1)\Gamma_{q}(\gamma+5)-10q^{2}\Gamma_{q}^{2}(\gamma+3)\right) },\]
	which is desired result.
	\item [\bf c.] If we take $\alpha=0$ in the thirth part of the \Cref{Starlikenessq-ML} we obtain that the radius of starlikeness of order zero of the function $h_{\gamma,\sigma}$ is the smallest positive root of the equation $(z\lambda_{\gamma,\sigma}(\sqrt{z};q))'=0.$ Therefore, we shall focus on the first positive zero of
	\begin{equation}\label{Theo2Partc}
    \varphi_{\gamma,\sigma}(z;q)=(z\lambda_{\gamma,\sigma}(\sqrt{z};q))'=\sum_{n\geq 0}\frac{(-1)^{n}\sigma^{2n}(n+1)q^{n(n-1)}}{\Gamma_{q}(2n+\gamma+1)}z^{n}.
	\end{equation}
	We know that the function $z\mapsto\lambda_{\gamma,\sigma}(z;q)$ belongs to the Laguerre-P\'olya class $\mathcal{LP},$ and cosequently we conclude that $z\mapsto \varphi_{\gamma,\sigma}(z;q)$ belongs also to the class Laguerre-P\'olya class. This means that the zeros of the function $z\mapsto \varphi_{\gamma,\sigma}(z;q)$ are all real. Suppose that $\varsigma_{\gamma,\sigma,n}(q)$ is the $n$th positive zero of the function $z\mapsto \varphi_{\gamma,\sigma}(z;q).$ Then infinite product representation of the function $z\mapsto \varphi_{\gamma,\sigma}(z;q)$ can be written as
	\begin{equation}\label{Theo2h2}
	\varphi_{\gamma,\sigma}(z;q)=\frac{1}{\Gamma_{q}(\gamma+1)}\prod_{n\geq 1}\left(1-\frac{z}{\varsigma_{\gamma,\sigma,n}(q)} \right) 
	\end{equation}
	Logarithmic differentiation of both sides of \eqref{Theo2h2} for $\left|z\right|<\varsigma_{\gamma,\sigma,1}(q)$ yields
	\begin{equation}\label{Theo2PartcCom1}
	\frac{\varphi_{\gamma,\sigma}'(z;q)}{\varphi_{\gamma,\sigma}(z;q)}=-\sum_{n\geq 1}\frac{1}{\varsigma_{\gamma,\sigma,n}(q)-z}=-\sum_{n\geq 1}\sum_{k\geq 0}\frac{z^k}{\varsigma_{\gamma,\sigma,n}^{k+1}(q)}=-\sum_{k\geq 0}\sum_{n\geq 1}\frac{z^k}{\varsigma_{\gamma,\sigma,n}^{k+1}(q)}=-\sum_{k\geq 0}\delta_{k+1}z^k,
	\end{equation}
	where $\delta_{k}=\sum_{n\geq 1}\varsigma_{\gamma,\sigma,n}^{-k}(q).$ On the other hand, with the aid of \eqref{Theo2Partc} we obtain
	\begin{equation}\label{Theo2PartcCom2}
	\frac{\varphi_{\gamma,\sigma}'(z;q)}{\varphi_{\gamma,\sigma}(z;q)}=-\sum_{n\geq 0}u_nz^n \bigg/\sum_{n\geq 0}v_nz^n
	\end{equation}
	where
	\[u_n=\frac{(-1)^n\sigma^{2n+2}(n+1)(n+2)q^{n(n+1)}}{\Gamma_{q}(2n+\gamma+3)} \text{ \ and \ } v_n=\frac{(-1)^n\sigma^{2n}(n+1)q^{n(n-1)}}{\Gamma_{q}(2n+\gamma+3)}.\]
	By comparing the coefficients of \eqref{Theo2PartcCom1} and \eqref{Theo2PartcCom2}, we arrive at 
	\[u_0=v_0\delta_1 \text{ \ and \ } u_1=v_0\delta_2+v_1\delta_1\]
	which give the following Rayleigh sums
	\[\delta_1=\frac{2\sigma^2\Gamma_{q}(\gamma+1)}{\Gamma_{q}(\gamma+3)} \text{ \ and \ } \delta_2=\frac{4\sigma^4\Gamma_{q}^2(\gamma+1)}{\Gamma_{q}^2(\gamma+3)}-\frac{6q^2\sigma^4\Gamma_{q}(\gamma+1)}{\Gamma_{q}(\gamma+5)}.\]
	By using the Euler-Rayleigh inequalities
	\[ \delta_{k}^{-\frac{1}{k}}< \varsigma_{\gamma,\sigma,1}(q) <\frac{\delta_{k}}{\delta_{k+1}}\]
	for $k=1$ we obtain
	\[\frac{\Gamma_{q}(\gamma+3)}{2\sigma^2\Gamma_{q}(\gamma+1)}<r^{\star}(h_{\gamma,\sigma}(z;q))<\frac{\Gamma_{q}(\gamma+3)\Gamma_{q}(\gamma+5)}{\sigma^2\left(2\Gamma_{q}(\gamma+1)\Gamma_{q}(\gamma+5)-3q^2\Gamma_{q}^2(\gamma+3)\right) }.\]
\end{itemize}
\end{proof}

\subsection{Radii of convexity of the $q-$Mittag-Leffler functions} This section is devoted to determine the radii of convexity of order $\alpha$ of the functions $f_{\gamma,\sigma}(z;q),$ $g_{\gamma,\sigma}(z;q)$ and $h_{\gamma,\sigma}(z;q).$ In addition, we find tight lower and upper bounds for the radii of convexity of order zero for the functions $f_{\gamma,\sigma}(z;q),$ $g_{\gamma,\sigma}(z;q)$ and $h_{\gamma,\sigma}(z;q).$

\begin{theorem}\label{Theo3}
Let $\alpha \in \left[0,1 \right) $ and with the conditions of \Cref{q-Mittag-LefflerLemma} the following assrtions are valid:
\begin{itemize}
	\item [\bf a.] The radius of convexity $r^{c}_{\alpha}\left(f_{\gamma,\sigma}(z;q) \right)$ is the smallest positive root of the transcendental equation
	\[\left( rf_{\gamma,\sigma}(z;q)\right)'=\alpha f_{\gamma,\sigma}'(z;q). \]
	\item [\bf b.] The radius of convexity $r^{c}_{\alpha}\left(g_{\gamma,\sigma}(z;q) \right)$ is the smallest positive root of the transcendental equation
	\[\left( rg_{\gamma,\sigma}(z;q)\right)'=\alpha g_{\gamma,\sigma}'(z;q). \]
	\item [\bf c.] The radius of convexity $r^{c}_{\alpha}\left(h_{\gamma,\sigma}(z;q) \right)$ is the smallest positive root of the transcendental equation
	\[\left( rg_{\gamma,\sigma}(z;q)\right)'=\alpha g_{\gamma,\sigma}'(z;q). \]
\end{itemize}
\end{theorem}
\begin{proof}
\begin{itemize}
	\item [\bf a.] It is easy to check that
	\[1+\frac{zf_{\gamma,\sigma}''(z;q) }{f_{\gamma,\sigma}'(z;q) }=1+\frac{z\Psi_{\gamma,\sigma}''(z;q)}{\Psi_{\gamma,\sigma}'(z;q)}+\left(\frac{1}{\gamma+1} -1\right)\frac{z\Psi_{\gamma,\sigma}'(z;q)}{\Psi_{\gamma,\sigma}(z;q)} .\]
	Moreover, from proof of \Cref{Theo1} we conclude the following infinite product representations
	\[\Psi_{\gamma,\sigma}(z;q)=\frac{z^{\gamma+1}}{\Gamma_{q}(\gamma+1)}\prod_{n\geq 1}\left(1-\frac{z^2}{\varepsilon_{\gamma,\sigma,n}^2(q)} \right) \text{ \ and \ } \Psi_{\gamma,\sigma}'(z;q)=\frac{(\gamma+1)z^{\gamma}}{\Gamma_{q}(\gamma+1)}\prod_{n\geq 1}\left(1-\frac{z^2}{\xi_{\gamma,\sigma,n}^2(q)}\right),\]
	where $\varepsilon_{\gamma,\sigma,n}(q)$ and $\xi_{\gamma,\sigma,n}(q)$ stand for the $n$th positive roots of $z\mapsto\Psi_{\gamma,\sigma}(z;q)$ and $z\mapsto\Psi_{\gamma,\sigma}'(z;q)$, respectively, as in \Cref{q-Mittag-LefflerLemma}. Logarithmic differentiation of both sides of the above infinite representations yields
	\[\frac{z\Psi_{\gamma,\sigma}'(z;q)}{\Psi_{\gamma,\sigma}(z;q)}=\gamma+1-\sum_{n\geq 1}\frac{2z^2}{\varepsilon_{\gamma,\sigma,n}^2(q)-z^2} \text{ \ and \ } \frac{z\Psi_{\gamma,\sigma}''(z;q)}{\Psi_{\gamma,\sigma}'(z;q)}=\gamma-\sum_{n\geq 1}\frac{2z^2}{\xi_{\gamma,\sigma,n}^2(q)-z^2},\]
	which gives
	\[1+\frac{zf_{\gamma,\sigma}''(z;q) }{f_{\gamma,\sigma}'(z;q) }=1-\left( \frac{1}{\gamma+1}-1 \right)\sum_{n\geq 1}\frac{2z^2}{\xi_{\gamma,\sigma,n}^2(q)-z^2}-\sum_{n\geq 1}\frac{2z^2}{\varepsilon_{\gamma,\sigma,n}^2(q)-z^2}.\]
	With the aid of the following inequlaity \cite[Lemma 2.1]{BS}
	\[\alpha\real\left(\frac{z}{a-z} \right)-\real\left( \frac{z}{b-z} \right) \geq \alpha\frac{\left|z\right|  }{a-\left| z\right| } -\frac{\left| z\right| }{b-\left| z\right| },\]
	where $a>b>0,$ $\alpha\in[0,1],$ $z\in\mathbb{C}$ we obtain for $\gamma\in\left(0,2 \right) $
	\begin{equation}
	\real\left(1+\frac{zf_{\gamma,\sigma}''(z;q) }{f_{\gamma,\sigma}'(z;q) } \right) \geq 1-\left( \frac{1}{\gamma+1}-1\right)\sum_{n\geq 1}\frac{2r^2}{\varepsilon_{\gamma,\sigma,n}^2(q)-r^2}-\sum_{n\geq 1}\frac{2r^2}{\xi_{\gamma,\sigma,n}^2(q)-r^2},
	\end{equation}
	 for all $z\in\mathbb{D}_{\xi_{\gamma,\sigma,1}}.$ It is important to mention that we used tacitly that the zeros of $\varepsilon_{\gamma,\sigma,n}(q)$ and $\xi_{\gamma,\sigma,n}(q)$ interlace, due to \Cref{q-Mittag-LefflerLemma}. In addition, the above deduced inequalities imply for $r\in\left( 0,\xi_{\gamma,\sigma,1}(q)\right)$
	 \[\inf_{z\in\mathbb{D}_r}\left\lbrace\real\left(  1+\frac{zf_{\gamma,\sigma}''(z;q) }{f_{\gamma,\sigma}'(z;q) }\right) \right\rbrace=1+\frac{rf_{\gamma,\sigma}''(r;q) }{f_{\gamma,\sigma}'(r;q) }.\]
	 The function $u_{\gamma,\sigma}:\left(0,\xi_{\gamma,\sigma,1} \right)\rightarrow\mathbb{R} $ given by
	 \[u_{\gamma,\sigma}(r;q)=1+\frac{rf_{\gamma,\sigma}''(r;q) }{f_{\gamma,\sigma}'(r;q) },\]
	 is strictly decreasing since
	 \begin{align*}
	 u_{\gamma,\sigma}'(r;q)&=-\left( \frac{1}{\gamma+1}-1\right)\sum_{n\geq 1}\frac{4r\varepsilon_{\gamma,\sigma,n}^2(q)}{\left(\varepsilon_{\gamma,\sigma,n}^2(q)-r^2 \right)^2 }-\sum_{n\geq 1}\frac{4r\xi_{\gamma,\sigma,n}^2(q)}{\left( \xi_{\gamma,\sigma,n}^2(q)-r^2\right)^2 }\\
	 &<\sum_{n\geq1}\frac{4r\varepsilon_{\gamma,\sigma,n}^2(q)}{\left(\varepsilon_{\gamma,\sigma,n}^2(q)-r^2 \right)^2 }-\sum_{n\geq 1}\frac{4r\xi_{\gamma,\sigma,n}^2(q)}{\left( \xi_{\gamma,\sigma,n}^2(q)-r^2\right)^2 }<0
	 \end{align*}
	 for $z\in(0,\xi_{\gamma,\sigma,1}(q)),$ where we used again the interlacing property of the zeros stated in \Cref{q-Mittag-LefflerLemma}. Observe also that $\lim_{r\searrow 0}u_{\gamma,\sigma}(r;q)=1$ and $\lim_{r\nearrow \xi_{\gamma,\sigma,1}}=-\infty,$ which means that for $z\in\mathbb{D}_{r_1}$ we get
	 \[\real\left(  1+\frac{zf_{\gamma,\sigma}''(z;q) }{f_{\gamma,\sigma}'(z;q) }\right)>\alpha \]
	 if and only if $r_1$ is unique root of 
	 \[1+\frac{zf_{\gamma,\sigma}''(r;q) }{f_{\gamma,\sigma}'(r;q) }=\alpha\]
	 situated in $\left(0,\xi_{\gamma,\sigma,1} \right).$
	 \item [\bf b.] By virtue of \eqref{Theo2PhiProd} we have
	 \[g_{\gamma,\sigma}'(z;q)=\prod_{n\geq1}\left(1-\frac{z^2}{\theta_{\gamma,\sigma,n}^2(q)}\right).\]
	 Now, taking logarithmic derivatives on both sides, we obtain
	 \[1+\frac{zg_{\gamma,\sigma}''(z;q)}{g_{\gamma,\sigma}'(z;q)}=1-\sum_{n\geq1}\frac{2z^2}{\theta_{\gamma,\sigma,n}^2(q)-z^2}.\]
	 In light of the inequality \eqref{ClassicalIneq} we get
	 \[\real\left(1+\frac{zg_{\gamma,\sigma}''(z;q)}{g_{\gamma,\sigma}'(z;q)} \right)\geq1-\sum_{n\geq1}\frac{2r^2}{\theta_{\gamma,\sigma,n}^2(q)-r^2}, \]
	 where $\left| z\right| =r.$ Hence for $r\in(0,\theta_{\gamma,\sigma,1}(q))$ we obtain
	 \[\inf_{z\in\mathbb{D}_r}\left\lbrace\real\left(1+\frac{zg_{\gamma,\sigma}''(z;q)}{g_{\gamma,\sigma}'(z;q)}\right)\right\rbrace=1+\frac{rg_{\gamma,\sigma}''(r;q)}{g_{\gamma,\sigma}'(r;q)}.\]
	 The function $v_{\gamma,\sigma}:\left(0,\theta_{\gamma,\sigma,1}(q) \right)\rightarrow \mathbb{R}, $ defined by
	 \[v_{\gamma,\sigma}(r;q)=1+\frac{rg_{\gamma,\sigma}''(r;q)}{g_{\gamma,\sigma}'(r;q)},\]
	 is strictly decreasing and take limits $\lim_{r\searrow 0}v_{\gamma,\sigma}(r;q)=1$ and $\lim_{r\nearrow \theta_{\gamma,\sigma,1}}v_{\gamma,\sigma}(r;q)=-\infty$ that means that for $z\in\mathbb{D}_{r_2}$ we get
	 \[\real\left(1+\frac{zg_{\gamma,\sigma}''(z;q)}{g_{\gamma,\sigma}'(z;q)} \right)>\alpha \]
	 if and only if $r_2$ is the unique root of 
	 \[1+\frac{zg_{\gamma,\sigma}''(z;q)}{g_{\gamma,\sigma}'(z;q)}=\alpha\]
	 situated in $\left( 0,\theta_{\gamma,\sigma,1}(q)\right).$
	 \item [\bf c.] By virtue of \eqref{Theo2h2} we have
	 \[h_{\gamma,\sigma}'(z;q)=\prod_{n\geq 1}\left(1-\frac{z}{\varsigma_{\gamma,\sigma,n}(q)} \right). \]
	 If we taking logarithmic derivatives on both sides, we obtain
	 \[1+\frac{zh_{\gamma,\sigma}''(z;q)}{h_{\gamma,\sigma}'(z;q)}=1-\sum_{n\geq1}\frac{z}{\varsigma_{\gamma,\sigma,n}-z}.\]
	 Let $r\in\left( 0,\varsigma_{\gamma,\sigma,1}\right) $ be a fixed number. The minimum principle for harmonic functions and inequality \eqref{ClassicalIneq} imply that for $z\in\mathbb{D}_{r}$ we have
	 \begin{align*}
	 \real\left(1+\frac{zh_{\gamma,\sigma}''(z;q)}{h_{\gamma,\sigma}'(z;q)} \right)&=\real\left(1-\sum_{n\geq1}\frac{z}{\varsigma_{\gamma,\sigma,n}-z} \right)\geq \min_{\left| z\right| =r}\real\left( 1-\sum_{n\geq1}\frac{z}{\varsigma_{\gamma,\sigma,n}-z}\right) \\
	 &= \min_{\left| z\right| =r}\left( 1-\sum_{n\geq1}\real\frac{z}{\varsigma_{\gamma,\sigma,n}-z}\right)\geq1-\sum_{n\geq1}\frac{r}{\varsigma_{\gamma,\sigma,n}(q)-r}\\
	 &=1+\frac{rh_{\gamma,\sigma}''(r;q)}{h_{\gamma,\sigma}'(r;q)}.
	 \end{align*}
	 Consequently, it follows that
	 \[\inf_{z\in\mathbb{D}_r}\left\lbrace \real\left( 1+\frac{zh_{\gamma,\sigma}''(z;q)}{h_{\gamma,\sigma}'(z;q)} \right)\right\rbrace =1+\frac{rh_{\gamma,\sigma}''(r;q)}{h_{\gamma,\sigma}'(r;q)}.\]
	 Now, let $r_3$ be the smallest positive root of the equation
	 \begin{equation}\label{Theo2hSon}
	 1+\frac{rh_{\gamma,\sigma}''(r;q)}{h_{\gamma,\sigma}'(r;q)}=\alpha.
	 \end{equation}
	 For $z\in\mathbb{D}_{r_3},$ we have
	 \[\real\left(1+\frac{zh_{\gamma,\sigma}''(z;q)}{h_{\gamma,\sigma}'(z;q)} \right) >\alpha.\]
	 In order to finish the proof, we need to show that equation \eqref{Theo2hSon} has a unique root in $(0,\varsigma_{\gamma,\sigma,1}(q)).$ But equation \eqref{Theo2hSon} is equivalent to
	 \[w_{\gamma,\sigma}(r;q)=1-\alpha-\sum_{n\geq 1}\frac{r}{\varsigma_{\gamma,\sigma,n}(q)-r}=0,\]
	 and we have
	 \[\lim_{r\searrow 0}w_{\gamma,\sigma}(r;q)=1-\alpha>0, \quad \lim_{r\nearrow \varsigma_{\gamma,\sigma,1}}w_{\gamma,\sigma}(r;q)=-\infty.\]
	 Since the function $w_{\gamma,\sigma}(r;q)$ is strictly decreasing on  $(0,\varsigma_{\gamma,\sigma,1}(q)),$ it follows that the equation $w_{\gamma,\sigma}(r;q)=0$ has a unique root.
\end{itemize}
\end{proof}
The following theorem gives some tight lower and upper bounds for the radii of convexity of the functions seen on the above theorem, that is of $g_{\gamma,\sigma}(z;q)$ and $h_{\gamma,\sigma}(z;q).$
\begin{theorem} \label{Theo4}
With the same conditions of \Cref{q-Mittag-LefflerLemma} the following inequalities are valid:
	\begin{itemize}
		\item [\bf a.] The radius of convexity $r^{c}\left(g_{\gamma,\sigma}(z;q) \right) $ satisfies the inequalities
		\[ \frac{\Gamma_{q}(\gamma+3)}{9\sigma^2\Gamma_{q}(\gamma+1)}<\left( r^{c}\left(g_{\gamma,\sigma}(z;q) \right) \right)^2<\frac{9\Gamma_{q}(\gamma+3)\Gamma_{q}(\gamma+5)}{\sigma^2\left(81\Gamma_{q}(\gamma+1)\Gamma_{q}(\gamma+5)-50q^2\Gamma_{q}^2(\gamma+3) \right) }.\]
		\item [\bf b.] The radius of convexity $r^{c}\left(h_{\gamma,\sigma}(z;q) \right) $ satisfies the inequalities
		\[\frac{\Gamma_{q}(\gamma+3)}{4\sigma^2\Gamma_{q}(\gamma+1)}<r^{c}(h_{\gamma,\sigma}(z;q))<\frac{2\Gamma_{q}(\gamma+3)\Gamma_{q}(\gamma+5)}{\sigma^2\left( 8\Gamma_{q}(\gamma+1)\Gamma_{q}(\gamma+5)-9q^2\Gamma_{q}^2(\gamma+3)\right) }.\]
	\end{itemize}
\end{theorem}
\begin{proof}
	\begin{itemize}
		\item [\bf a.] By using the infinite series representations of the  $q-$Mittag-Leffler function and its derivative we obtain
		\begin{align*}
		\Phi_{\gamma,\sigma}(z;q)&=\left(zg_{\gamma,\sigma}'(z;q)\right)'=\Gamma_{q}(\gamma+1)\sum_{n\geq 0}\frac{(-1)^n\sigma^{2n}(2n+1)^{2}q^{n(n-1)}}{\Gamma_{q}(2n+\gamma+1)}z^{2n}\\&=1+\sum_{n\geq 1}\frac{(-1)^n\sigma^{2n}(2n+1)^{2}q^{n(n-1)}}{\Gamma_{q}(2n+\gamma+1)}z^{2n}.
		\end{align*}
		We know that $g_{\gamma,\sigma}(z;q)\in\mathcal{LP}$ and this in turn implies that $z\mapsto\Phi_{\gamma,\sigma}(z;q)$ belongs also to the the Laguerre-P\'olya class and consequently all its zeros are real. Suppose that $\ell_{\gamma,\sigma,n}(q)$ is the $n$th positive zero of the function $z\mapsto\Phi_{\gamma,\sigma}(z;q).$ Then we deduce that
		\begin{equation}\label{Theo41}
		\Phi_{\gamma,\sigma}(z;q)=\prod_{n\geq 1}\left(1-\frac{z^2}{\ell_{\gamma,\sigma,n}^2(q)} \right).
		\end{equation}
		Logarithmic differentiation of both sides of \eqref{Theo41} implies for $\left| z\right| <\ell_{\gamma,\sigma,1}(q)$
		\begin{equation}\label{Theo4g2}
		\text{\footnotesize $ \frac{\Phi_{\gamma,\sigma}'(z;q)}{\Phi_{\gamma,\sigma}(z;q)}=\sum_{n\geq 1}\frac{-2z}{\ell_{\gamma,\sigma,n}^2(q)-z^2}=\sum_{n\geq1}\sum_{k\geq 0}\frac{-2z^{2k+1}}{\ell_{\gamma,\sigma,n}^{2k+2}(q)}=\sum_{k\geq 0}\sum_{n\geq 1}\frac{-2z^{2k+1}}{\ell_{\gamma,\sigma,n}^{2k+2}(q)}=-2\sum_{k\geq 0}\mu_{k+1}z^{2k+1},$}
		\end{equation}
		where, $\mu_{k}=\sum_{n\geq 1}\ell_{\gamma,\sigma,n}^{-2k}(q).$ On the other hand, we have
		\begin{equation}\label{Theo4g3}
		\frac{\Phi_{\gamma,\sigma}'(z;q)}{\Phi_{\gamma,\sigma}(z;q)}=\sum_{n\geq 0}r_nz^{2n+1} \bigg/ \sum_{n\geq 0}s_nz^{2n}
		\end{equation}
		where
		\[r_n=\frac{(-1)^n\sigma^{2n+2}(n+1)(2n+3)^{2}q^{n(n+1)}}{\Gamma_{q}(2n+\gamma+3)} \text{ \ and \ }s_n=\frac{(-1)^n\sigma^{2n}(2n+1)^2q^{n(n-1)}}{\Gamma_{q}(2n+\gamma+1)}. \]
		By comparing the coefficients of \eqref{Theo4g2} and \eqref{Theo4g3} we have
		\[r_0=\mu_{1}s_0 \text{ \ and \ } r_1=s_0\mu_{2}+s_1\mu_{1}\]
		which give us the following Rayleigh sums
		\[\mu_{1}=\frac{9\sigma^2\Gamma_{q}(\gamma+1)}{\Gamma_{q}(\gamma+3)} \text{ \ and \ } \mu_{2}=\frac{81\sigma^4\Gamma_{q}^2(\gamma+1)}{\Gamma_{q}^2(\gamma+3)}-\frac{50\sigma^4q^2\Gamma_{q}(\gamma+1)}{\Gamma_{q}(\gamma+5)}.\]
		By using the Euler-Rayleigh inequalities
		\[\mu_{k}^{-\frac{1}{k}}<\ell_{\gamma,\sigma,1}^2(q)<\frac{\mu_{k}}{\mu_{k+1}}\]
		for $k=1$ we obtain the following inequalities
		\[ \frac{\Gamma_{q}(\gamma+3)}{9\sigma^2\Gamma_{q}(\gamma+1)}<\left( r^{c}\left(g_{\gamma,\sigma}(z;q) \right) \right)^2<\frac{9\Gamma_{q}(\gamma+3)\Gamma_{q}(\gamma+5)}{\sigma^2\left(81\Gamma_{q}(\gamma+1)\Gamma_{q}(\gamma+5)-50q^2\Gamma_{q}^2(\gamma+3) \right) },\]
		which is desired result.
		\item [\bf b.] By means of the definition of the $q-$Mittag-Leffler function we have
		\[\psi_{\gamma,\sigma}(z;q)=\left( zh_{\gamma,\sigma}'(z;q)\right)'=1+\sum_{n\geq 1}\frac{(-1)^n\sigma^{2n}(n+1)^2q^{n(n-1)}}{\Gamma_{q}(2n+\gamma+1)}z^n \]
		and consequently
		\begin{equation}\label{Theo4h}
		\frac{\psi_{\gamma,\sigma}'(z;q)}{\psi_{\gamma,\sigma}(z;q)}=-\sum_{n\geq 0}t_nz^n \bigg/ \sum_{n\geq 0}p_nz^n
		\end{equation}
		where
		\[t_n=\frac{(-1)^n\sigma^{2n+2}(n+1)(n+2)^2q^{n(n+1)}}{\Gamma_{q}(2n+\gamma+3)} \text{ \ and \ } p_n=\frac{(-1)^n\sigma^{2n}(n+1)^2q^{n(n-1)}}{\Gamma_{q}(2n+\gamma+1)}.\]
		Because of the fact that $h_{\gamma,\sigma}(z;q)$ belongs to the Laguerre-P\'olya class $\mathcal{LP},$ it follows that $h_{\gamma,\sigma}'(z;q)\in\mathcal{LP},$ and consequently the function $z\mapsto\psi_{\gamma,\sigma}(z;q)$ belongs also to the Laguerre-P\'olya class $\mathcal{LP},$ and hence all its zeros are real. Assume that $\nu_{\gamma,\sigma,n}(q)$ is the $n$th positive zero of the function $z\mapsto\psi_{\gamma,\sigma}(z;q),$ then the following infinite product representation take place
		\begin{equation}\label{Theo4h1}
		\psi_{\gamma,\sigma}(z;q)=\prod_{n\geq 1}\left(1-\frac{z}{\nu_{\gamma,\sigma,n}(q)} \right).
		\end{equation}
		Logarithmic differentiation of both sides of \eqref{Theo4h1} implies for $\left| z\right| <\nu_{\gamma,\sigma,1}(q)$
		\begin{equation}\label{Theo4h2}
		\text{\footnotesize $\frac{\psi_{\gamma,\sigma}'(z;q)}{\psi_{\gamma,\sigma}(z;q)}=-\sum_{n\geq 1}\frac{1}{\nu_{\gamma,\sigma,n}(q)-z}=-\sum_{n\geq 1}\sum_{k\geq 0}\frac{z^k}{\nu_{\gamma,\sigma,n}^{k+1}(q)}=-\sum_{k\geq 0}\sum_{n\geq 1}\frac{z^k}{\nu_{\gamma,\sigma,n}^{k+1}(q)}=-\sum_{k\geq 0}\varrho_{k+1}z^k, $}
		\end{equation}
		where $\varrho_{k}=\sum_{n\geq 1}\nu_{\gamma,\sigma,n}^{-k}(q).$ By comparing the coefficients of \eqref{Theo4h} and \eqref{Theo4h2} we obtain the following Rayleigh sums
		\[\varrho_1=\frac{4\sigma^2\Gamma_{q}(\gamma+1)}{\Gamma_{q}(\gamma+3)} \text{ \ and \ } \varrho_2=\frac{16\sigma^4\Gamma_{q}^2(\gamma+1)}{\Gamma_{q}^2(\gamma+3)}-\frac{18\sigma^4q^2\Gamma_{q}(\gamma+1)}{\Gamma_{q}(\gamma+5)}.\]
		By making use of the Euler-Rayleigh inequalities
		\[\varrho_{k}^{-\frac{1}{k}}<\nu_{\gamma,\sigma,1}(q)<\frac{\varrho_{k}}{\varrho_{k+1}}\]
		for $k=1$ the following inequalities immediately take place
		\[\frac{\Gamma_{q}(\gamma+3)}{4\sigma^2\Gamma_{q}(\gamma+1)}<r^{c}(h_{\gamma,\sigma}(z;q))<\frac{2\Gamma_{q}(\gamma+3)\Gamma_{q}(\gamma+5)}{\sigma^2\left( 8\Gamma_{q}(\gamma+1)\Gamma_{q}(\gamma+5)-9q^2\Gamma_{q}^2(\gamma+3)\right) },\]
		which is desired result.
	\end{itemize}
\end{proof}
\subsection{Some particular cases of the main results} This section is devoted to give some interesting results corresponding with the main results for some particular cases, in particular for $\gamma=0$ and $\gamma=1.$ It is important to note that for $\gamma=0,$ the radii of starlikeness and convexity of the functions $f_{0,\sigma}(z;q)$  and $g_{0,\sigma}(z;q)$ coincide with each other.

If we take $\gamma=0$ in \Cref{Theo1}, we arrive at the following results, respectively. 
\begin{corollary}
Let $\alpha \in \left[0,1 \right) $ and with the conditions of \Cref{q-Mittag-LefflerLemma} the following assertions hold true:
\begin{itemize}
	\item [\bf a.] The radius of starlikeness of order $\alpha$ of the function  $f_{0,\sigma}(z;q)=g_{0,\sigma}(z;q)=z\cos (q^{-\frac{1}{2}}\sigma z;q)$ is $r^{\star}_{\alpha}( f_{0,\sigma}(z;q))=x_{0,\sigma,1}(q)$, where $x_{0,\sigma,1}(q)$ stands for the smallest positive zero of the equation
	\[rq^{-\frac{1}{2}}\sigma\sin(q^{-\frac{1}{2}}\sigma z;q) +(\alpha-1)\cos (q^{-\frac{1}{2}}\sigma z;q)=0.\]
	\item [\bf b.] The radius of starlikeness of order $\alpha$ of the function  $h_{0,\sigma}=z\cos (q^{-\frac{1}{2}}\sigma \sqrt{z};q)$ is $r^{\star}_{\alpha}( h_{0,\sigma}(z;q))=z_{0,\sigma,1}(q)$, where $z_{0,\sigma,1}$ stands for the smallest positive zero of the equation
	\[\frac{\sigma}{1+\sqrt{q}}\sin (q^{-\frac{1}{2}}\sigma \sqrt{z};q)+2(\alpha-1)\cos(q^{-\frac{1}{2}}\sigma \sqrt{z};q)=0.\]
\end{itemize}
\end{corollary}
Putting $\gamma=0$ in \Cref{Starlikenessq-ML} we have the following results.
\begin{corollary}
Let the conditions of \Cref{q-Mittag-LefflerLemma} remain valid.
	\begin{itemize}
		\item [\bf a.] The radii of starlikeness $r^{\star}(f_{0,\sigma}(z;q))$ and $r^{\star}(g_{0,\sigma}(z;q))$ satisfies the inequalities
		\[\frac{1+q}{3\sigma^2}<\left(r^{\star}(f_{0,\sigma}(z;q)) \right)^2< \frac{3(1+q)(1+q^2)(1+q+q^2)}{\sigma^2(9q^4+9q^3+8q^2+9q+9)} .\]
		\item [\bf b.] The radius of starlikeness $r^{\star}(h_{\gamma,\sigma}(z;q))$ satisfies the inequalities
		\[\frac{1+q}{2\sigma^2}<r^{\star}(h_{\gamma,\sigma}(z;q))<\frac{(1+q)(1+q^2)(1+q+q^2)}{\sigma^2(2q^4+2q^3+q^2+2q+2)}.\]
	\end{itemize}
\end{corollary}
Setting $\gamma=0$ in \Cref{Theo3} we get the following results.
\begin{corollary}
With the same conditions of \Cref{q-Mittag-LefflerLemma} the following inequalities are valid:
\begin{itemize}
	\item [\bf a.] The radius of convexity $r^{c}\left(g_{0,\sigma}(z;q) \right) $ satisfies the inequalities
	\[ \frac{1+q}{9\sigma^2}<\left( r^{c}\left(g_{0,\sigma}(z;q) \right) \right)^2<\frac{9(1+q)(1+q^2)(1+q+q^2)}{\sigma^2(81q^4+81q^3+112q^2+81q+81)}.\]
	\item [\bf b.] The radius of convexity $r^{c}\left(h_{0,\sigma}(z;q) \right) $ satisfies the inequalities
	\[\frac{1+q}{4\sigma^2}<r^{c}(h_{0,\sigma}(z;q))<\frac{2(1+q)(1+q^2)(1+q+q^2)}{\sigma^2(8q^4+8q^3+7q^2+8q+8)}.\]
\end{itemize}
\end{corollary}

\end{document}